\documentclass[12pt]{amsart}

\usepackage{amssymb, amsmath, amsthm}

\allowdisplaybreaks

\numberwithin{equation}{section}

\theoremstyle{plain}
\newtheorem{thm}{Theorem}[section]
\newtheorem{prop}[thm]{Proposition}
\newtheorem{lem}[thm]{Lemma}

\theoremstyle{definition}

\newtheorem{rem}[thm]{Remark}

\newcommand{\ichi}{\mathbf{1}}

\newcommand{\K}{\mathbb{K}}
\newcommand{\N}{\mathbb{N}}
\newcommand{\R}{\mathbb{R}}

\newcommand{\Z}{\mathbb{Z}}

\newcommand{\calF}{\mathcal{F}}

\newcommand{\calS}{\mathcal{S}}
\newcommand{\supp}{\mathrm{supp}\, }

\newcommand{\II}{I\hspace{-3pt}I}

\begin{document}
\title[Bilinear pseudo-differential operators of the Sj\"ostrand class]
{A remark on bilinear pseudo-differential operators 
with symbols in the Sj\"ostrand class}

\author[T. Kato]{Tomoya Kato}

\address{Division of Pure and Applied Science, 
Faculty of Science and Technology, Gunma University, 
Kiryu, Gunma 376-8515, Japan}

\email{t.katou@gunma-u.ac.jp}

\date{\today}

\keywords{Bilinear pseudo-differential operators,
Sj\"ostrand symbol classes}
\thanks{This work was supported by 
the association for the
advancement of Science and Technology, Gunma University.}

\subjclass[2010]{35S05, 42B15, 42B35}

\begin{abstract}
In this short note, 
we consider bilinear pseudo-differential operators
with symbols belonging to the Sj\"ostrand class.
We show that those operators are bounded
from the product of the $L^2$-based 
Sobolev spaces $H^{s_1} \times H^{s_2}$ to $L^r$
for $s_1,s_2 > 0$, $s_1+s_2=n/2$,
and $1 \leq r \leq 2$.
\end{abstract}
\maketitle

\section{Introduction}\label{Introduction}

For a bounded measurable 
function $\sigma = \sigma (x, \xi_1, \xi_2)$ 
on $(\R^n)^{3}$,
the bilinear pseudo-differential operator
$T_{\sigma}$ is defined 
by
\[
T_{\sigma}(f_1,f_2)(x)
=\frac{1}{(2\pi)^{2n}}
\int_{(\R^n)^{2}}e^{i x \cdot(\xi_1+\xi_2)}
\sigma (x, \xi_1, \xi_2) 
\widehat{f_1}(\xi_1)\widehat{f_2}(\xi_2)
\, d\xi_1 d\xi_2
\]
for $f_1,f_2 \in \calS(\R^n)$. 
The bilinear H\"ormander symbol class,
$BS_{\rho, \delta}^m=BS_{\rho, \delta}^m(\R^n)$,
$m \in \R$, $0 \leq \delta \leq \rho \leq 1$,
consists of all
$\sigma(x,\xi_1,\xi_2) \in C^{\infty}((\R^n)^{3})$
such that
\[
|\partial^{\alpha}_x\partial^{\beta_1}_{\xi_1}
\partial^{\beta_2}_{\xi_2}\sigma(x,\xi_1,\xi_2)|
\le C_{\alpha,\beta_1,\beta_2}
(1+|\xi_1|+|\xi_2|)^{m+\delta|\alpha|-\rho(|\beta_1|+|\beta_2|)}
\]
for all multi-indices
$\alpha,\beta_1,\beta_2 \in \N_0^n=\{0, 1, 2, \dots \}^n$. 
When we state the boundedness of the bilinear operators 
$T_{\sigma}$, 
we will use the following terminology with a slight abuse.
Let $X_1,X_2$, and $Y$ be function spaces on $\R^n$ 
equipped with quasi-norms 
$\|\cdot \|_{X_1}$, $\|\cdot \|_{X_2}$, 
and $\|\cdot \|_{Y}$, 
respectively.  
If there exist a constant $C$ such that the estimate
\begin{equation*}
\| T_\sigma (f_1,f_2) \|_{Y} 
\leq
C
\| f_1 \|_{X_1}
\| f_2 \|_{X_2}
\end{equation*}
holds for all $f_1 \in \mathcal{S} \cap X_1$ 
and $f_2 \in \mathcal{S} \cap X_2$, 
then we simply say that
the operator $T_\sigma$ 
is bounded from $X_1 \times X_2$ to $Y$.

The interest of this short note is the boundedness 
from $L^2 \times L^2$ to $L^1$
for the bilinear operator $T_{\sigma}$
with the symbol $\sigma$ belonging 
to the Sj\"ostrand class.
We shall first recall some related boundedness 
results on the linear case.
For a bounded measurable 
function $\sigma = \sigma (x, \xi)$ on $(\R^n)^{2}$,
the linear pseudo-differential operator
$\sigma(X,D)$ is defined by
\begin{equation*}
\sigma (X,D) f (x) 
=
\frac{1}{\, ( 2 \pi )^n \,} 
\int_{\mathbb{R}^n} e^{i x \cdot \xi} 
\sigma (x , \xi ) \widehat f (\xi ) d\xi.
\end{equation*}
for $f \in \calS(\R^n)$.
The (linear) H\"ormander symbol class,
$S_{\rho, \delta}^m=S_{\rho, \delta}^m(\R^n)$,
$m \in \R$, $0 \leq \delta \leq \rho \leq 1$, 
consists of all functions 
$\sigma \in C^\infty ((\R^n)^{2})$ satisfying
\begin{equation*}
| \partial_x^\alpha \partial_\xi^\beta \sigma (x,\xi) | 
\leq C_{\alpha, \beta} (1+|\xi|)^{m + \delta |\alpha| - \rho |\beta|}
\end{equation*}
for all multi-indices $\alpha, \beta \in \N_0^n$.
In \cite{CV}, Calder\'on--Vaillancourt proved that
if symbols belong to the H\"ormander class $S^0_{0,0}$,
the linear pseudo-differential operators are bounded on $L^2$.
Then, Sj\"ostrand \cite{sjostrand 1994} introduced
a new wider class generating $L^2$-bounded pseudo-differential operators 
than the class $S^0_{0,0}$.
This new symbol class is today called 
as the Sj\"ostrand symbol class,
and is also identified 
as the modulation space $M^{\infty,1} ((\R^{n})^2)$
(see also Boulkhemair
\cite{boulkhemair 1995}). 
See Section \ref{secmodsp}
for the definition of modulation spaces.

We shall next consider the bilinear case.
Based on the linear case, 
one may expect the boundedness
for the bilinear pseudo-differential operators
of the bilinear H\"ormander class $BS^{0}_{0,0}$ and 
the Sj\"ostrand class $M^{\infty,1} ((\R^{n})^{3})$.
However, B\'enyi--Torres \cite{BT-2004}
pointed out that bilinear pseudo-differential 
operators with symbols in $BS^0_{0,0}$
are not bounded from $L^2 \times L^2 $ to $L^1$.
(See, e.g., \cite{MT-2013} for the boundedness 
for the bilinear H\"ormander class.)
Therefore, 
since $BS^0_{0,0} \hookrightarrow M^{\infty,1} ((\R^{n})^{3})$,
we never have the boundedness for the Sj\"ostrand class.
On the other hand,
B\'enyi--Gr\"ochenig--Heil--Okoudjou 
\cite{benyi grochenig heil okoudjou 2005}
proved that if $\sigma \in M^{\infty,1} ((\R^{n})^3)$,
then $T_\sigma$ is bounded from $L^2 \times L^2$ 
to the modulation space $M^{1,\infty}$, 
whose target space is wider than $L^1$.
Then, B\'enyi--Okoudjou 
\cite{benyi okoudjou 2004, benyi okoudjou 2006} 
gave that
if $\sigma$ belongs to $M^{1,1} ((\R^{n})^3)$,
embedded into $M^{\infty,1} ((\R^{n})^3)$,
then $T_\sigma$ is bounded from 
$L^{2} \times L^{2}$ to $L^{1}$.
According to these results, 
in the bilinear case, 
we are able to have the boundedness on 
$L^2 \times L^2$
for Sj\"ostrand symbol class,
paying some kind of cost.

Very recently, 
in \cite{KMT-arXiv},
it was proved that
if $\sigma \in BS_{0,0}^{0}$,
then the operator $T_\sigma$ is bounded from 
$H^{s_1} \times H^{s_2}$ to $(L^{2}, \ell^1)$
for $s_1,s_2>0$, $s_1+s_2=n/2$.
Here, $H^s$, $s\in\R$,
is the $L^2$-based Sobolev space
and $(L^{2}, \ell^1)$ 
is the $L^2$-based amalgam space
(see Section \ref{secpreli}).
The aim of this short note is
to improve this boundedness 
for the class $BS_{0,0}^{0}$
to that for the Sj\"ostrand class.
The main result is the following.

\begin{thm}\label{main-thm}
Let $s_1, s_2 \in (0,\infty)$ satisfy $s_1+s_2=n/2$.
Then, if $\sigma \in M^{\infty,1} ((\R^n)^3)$, 
the bilinear pseudo-differential operator 
$T_{\sigma}$ 
is bounded from 
$H^{s_1} \times H^{s_2}$ 
to $(L^2,\ell^{1})$. 
In particular, all those 
$T_{\sigma}$ are  
bounded from 
$H^{s_1} \times H^{s_2}$ 
to $L^r $ for all $r\in (1,2]$ and 
to $h^1 $. 
\end{thm}

We end this section by explaining the organization of this note. 
In Section \ref{secpreli}, we will give the basic notations 
which will be used throughout this paper
and recall the definitions and properties of some function spaces. 
In Section \ref{sectionLemma}, we collect some lemmas 
for the proof of Theorem \ref{main-thm}.
In Section \ref{sectionMain}, we show Theorem \ref{main-thm}.


\section{Preliminaries}\label{secpreli}

\subsection{Basic notations}

We collect notations which will be used throughout this paper.
We denote by $\R$, $\Z$, $\N$, and $\N_0$
the sets of real numbers, integers, positive integers, 
and nonnegative integers, respectively. 
We denote by $Q$ the $n$-dimensional 
unit cube $[-1/2,1/2)^n$.
The cubes $\tau + Q$, $\tau \in \Z^n$, are mutually 
disjoint and constitute a partition of the 
Euclidean space $\R^n$.
This implies integral of a function on $\R^{n}$ 
can be written as 
\begin{equation}\label{cubicdiscretization}
\int_{\R^n} f(x)\, dx = 
\sum_{\tau \in \Z^{n}} \int_{Q} f(x+ \tau)\, dx. 
\end{equation}
We denote by $B_R$ the closed ball in $\R^n$
of radius $R>0$ centered at the origin.
We write the characteristic function 
on the set $\Omega$ as $\mathbf{1}_{\Omega}$.
For $x\in \R^d$, we write 
$\langle x \rangle = (1 + | x |^2)^{1/2}$.

For two nonnegative functions $A(x)$ and $B(x)$ defined 
on a set $X$, 
we write $A(x) \lesssim B(x)$ for $x\in X$ to mean that 
there exists a positive constant $C$ such that 
$A(x) \le CB(x)$ for all $x\in X$. 
We often omit to mention the set $X$ when it is 
obviously recognized.  
Also $A(x) \approx B(x)$ means that
$A(x) \lesssim B(x)$ and $B(x) \lesssim A(x)$.

We denote the Schwartz space of rapidly 
decreasing smooth functions
on $\R^d$ 
by $\calS (\R^d)$ 
and its dual,
the space of tempered distributions, 
by $\calS^\prime(\R^d)$. 
The Fourier transform and the inverse 
Fourier transform of $f \in \calS(\R^d)$ are given by
\begin{align*}
\mathcal{F} f  (\xi) 
&= \widehat {f} (\xi) 
= \int_{\R^d}  e^{-i \xi \cdot x } f(x) \, dx, 
\\
\mathcal{F}^{-1} f (x) 
&= \check f (x)
= \frac{1}{(2\pi)^d} \int_{\R^d}  e^{i x \cdot \xi } f( \xi ) \, d\xi,
\end{align*}
respectively. 
For $m \in \calS^\prime (\R^d)$, 
the Fourier multiplier operator is defined by
\begin{equation*}
m(D) f 
=
\mathcal{F}^{-1} \big[ m \widehat{f}\, \big].
\end{equation*}
We also use the notation $(m(D)f)(x)=m(D_x)f(x)$ 
when we indicate which variable is considered.

For a measurable subset $E \subset \R^d$, 
the Lebesgue space $L^p (E)$, $1 \le p\le \infty$, 
is the set of all those 
measurable functions $f$ on $E$ such that 
$\| f \|_{L^p(E)} = 
\left( \int_{E} \big| f(x) \big|^p \, dx \right)^{1/p} 
< \infty 
$
if $1 \le p < \infty$ 
or   
$\| f \|_{L^\infty (E)} 
= 
\mathrm{ess}\, \sup_{x\in E} |f(x)|
< \infty$ if $p = \infty$. 
We also use the notation 
$\| f \|_{L^p(E)} = \| f(x) \|_{L^p_{x}(E)} $ 
when we want to indicate the variable explicitly.

The uniformly local $L^2$ space, denoted by 
$L^2_{ul} (\R^d)$, consists of 
all those measurable functions $f$ on 
$\R^d$ such that 
\begin{equation*}
\| f \|_{L^2_{ul} (\R^d) } 
= \sup_{\nu \in \Z^d}
\left( \int_{[-1/2, 1/2)^d} 
\big| f(x+\nu) 
\big|^2 \, dx 
\right)^{1/2}
< \infty 
\end{equation*}
(this notion can be found in 
\cite[Definition 2.3]{kato 1975}).

%

Let $\K$ be a countable set. 
We define the sequence spaces 
$\ell^q (\K)$ and $\ell^{q, \infty} (\K)$ 
as follows.  
The space $\ell^q (\K)$, $1 \le q \le \infty$, 
consists of all those 
complex sequences $a=\{a_k\}_{k\in \K}$ 
such that 
$ \| a \|_{ \ell^q (\K)} 
= 
\left( \sum_{ k \in \K } 
| a_k |^q \right)^{ 1/q } <\infty$ 
if $1 \leq q < \infty$  
or 
$\| a \|_{ \ell^\infty (\K)} 
= \sup_{k \in \K} |a_k| < \infty$ 
if $q = \infty$.
For $1\le q<\infty$, 
the space 
$\ell^{ q,\infty }(\K)$ is 
the set of all those complex 
sequences 
$a= \{ a_k \}_{ k \in \K }$ such that
\begin{equation*}
\| a \|_{\ell^{q,\infty} (\K) } 
=\sup_{t>0} 
\big\{ t\, 
\sharp \big( \{ k \in \K : | a_k | > t \} 
\big)^{1/q} \big\} < \infty,
\end{equation*}
where $\sharp$ denotes the cardinality 
of a set. 
Sometimes we write 
$\| a \|_{\ell^{q}} = 
\| a_k \|_{\ell^{q}_k }$ or 
$\| a \|_{\ell^{q,\infty}} = 
\| a_k \|_{\ell^{q,\infty}_k }$.  
If $\K=\Z^n$, 
we usually write $ \ell^q $ or $\ell^{ q, \infty }$ 
for 
$\ell^q (\Z^n)$  
or $\ell^{q, \infty} (\Z^n)$.

Let $X,Y,Z$ be function spaces.  
We denote the mixed norm by
\begin{equation*}
\label{normXYZ}
\| f (x,y,z) \|_{ X_x Y_y Z_z } 
= 
\bigg\| \big\| \| f (x,y,z) 
\|_{ X_x } \big\|_{ Y_y } \bigg\|_{ Z_z }.
\end{equation*} 
(Here pay special attention to the order 
of taking norms.)  
We shall use these mixed norms for  
$X, Y, Z$ being $L^p$ or $\ell^p$.


\subsection{Modulation spaces} \label{secmodsp}

We give the definition of modulation spaces which were introduced 
by Feichtinger \cite{feichtinger 1983, feichtinger 2006}
(see also Gr\"ochenig \cite{grochenig 2001}). 
Let $\varphi \in \calS(\R^d)$ satisfy that
$\supp \varphi \subset [-1,1]^{d}$ and
$\sum _{k\in \Z^{d}} \varphi (\xi - k) = 1$
for any $\xi \in \R^d$.
Then, for $1 \leq p,q \leq \infty$, 
the modulation space $M^{p,q}(\R^d)$ 
consists of
all $f \in \calS'(\R^d)$ such that 
\begin{equation*}
\| f \|_{M^{p,q}} = 
\big\|\varphi(D-k) f (x) \big\|_{L^{p}_{x}(\R^d) \ell^{q}_{k}(\Z^d) }
< \infty.
\end{equation*}

We note that the definition of modulation spaces 
is independent of the choice of the function $\varphi$.
If $p_{1}\leq p_{2}$ and
$q_{1}\leq q_{2}$, $M^{p_{1},q_{1}} \hookrightarrow M^{p_{2},q_{2}}$. 
We have $M^{2,2} = L^2$, 
and $M^{p,1} \hookrightarrow L^{p} \hookrightarrow M^{p,\infty }$ for $1 \leq p \leq \infty $. 
For more details, see also, e.g.,
\cite{kobayashi 2006, wang hudzik 2007}.

\subsection{Local Hardy space $h^1$} 
\label{sechardy}

We recall the definition of the local Hardy space 
$h^1(\R^n)$.

Let $\varphi \in \calS(\R^n)$ be such that
$\int_{\R^n}\varphi(x)\, dx \neq 0$. 
Then, the local Hardy space $h^1(\R^n)$ 
consists of
all $f \in \calS'(\R^n)$ such that 
$\|f\|_{h^1}= \|\sup_{0<t<1}|\varphi_t*f| \|_{L^1}
<\infty$,
where $\varphi_t(x)=t^{-n}\varphi(x/t)$.
It is known that $h^1(\R^n)$
does not depend on the choice of the function $\varphi$,
and that $h^1(\R^n) \hookrightarrow L^1(\R^n)$. 
See Goldberg \cite{goldberg 1979} for more details about 
$h^1$. 

\subsection{Amalgam spaces} 
\label{secamalgam}

For $1 \leq p,q \leq \infty$, 
the amalgam space 
$(L^p,\ell^q)(\R^n)$ is 
defined to be the set of all 
those measurable functions $f$ on 
$\R^n$ such that 
\begin{equation*}
\| f \|_{ (L^p,\ell^q) (\R^n)} 
=
 \| f(x+\nu) \|_{L^p_x (Q) \ell^q_{\nu}(\Z^n) }
=\left\{ \sum_{\nu \in \Z^n}
\left( \int_{Q} \big| f(x+\nu) \big|^p \, dx 
\right)^{q/p} \right\}^{1/q} 
< \infty  
\end{equation*}
with usual modification when $p$ or $q$ is infinity.  
Obviously, $(L^p,\ell^p) = L^p$ 
and $(L^2, \ell^\infty) = L^2_{ul}$. 
If $p_1 \geq p_2$ and $q_1 \leq q_2$, 
then 
$(L^{p_1}, \ell^{q_1}) 
\hookrightarrow (L^{p_2},\ell^{q_2}) $.
In particular, 
$(L^2,\ell^r) \hookrightarrow L^r$ 
for $1 \leq r \leq 2$. 
In the case $r=1$, the stronger embedding 
$(L^2,\ell^1) \hookrightarrow h^1$ holds
(see \cite[Section 2.3]{KMT-arxiv-2}). 
See Fournier--Stewart \cite{fournier stewart 1985} 
and Holland \cite{holland 1975} for more 
properties of amalgam spaces.
We end this subsection with noting the following,
which was proved in \cite[Lemma 2.1]{KMT-arxiv-2}.

\begin{lem}\label{lemequivalentamalgam}
Let $1 \leq p,q \leq \infty$. 
If $L> n/ \min (p,q)$ and if 
$g$ is a measuarable function on $\R^n$ such that 
\begin{equation*}
c \ichi_{Q} (x) \le |g(x)| \le c^{-1} \langle x \rangle ^{-L} 
\end{equation*}
with some positive constant $c$, then 
\begin{equation*}
\| f \|_{ (L^p,\ell^q) (\R^{n})} 
\approx 
\left\| 
g(x-\nu ) f(x) \right\|_{L^p_x(\R^{n}) \ell^q_{\nu}(\Z^{n}) }.
\end{equation*}
\end{lem}


\section{Lemmas}\label{sectionLemma}

In this section, we prepare several lemmas.
We denote by $S$ the operator
\begin{equation*}
S (f) (x) 
= \int_{\R^n} \frac{ |f(y)| }{ \langle x-y \rangle^{n+1} } dy.
\end{equation*}
We have the following facts,
which was proved in \cite[Lemmas 4.1 and 4.3]{KMT-arxiv-2}.

\begin{lem}\label{propertiesofS} 
Let $1 \leq p \leq \infty$.
The following (1)--(3) 
hold for all nonnegative measurable 
functions $f, g$ on $\R^n$. 
\begin{enumerate}
\setlength{\itemindent}{0pt} 
\setlength{\itemsep}{3pt} 
\item 
$S(f \ast g)(x) = \big( S(f) \ast g \big)(x) =  
\big( f \ast S(g) \big)(x)$.
\item 
$S(f)(x) \approx S(f)(y)$
for $x,y\in \R^n$ such that $|x-y|\lesssim1$.
\item 
$\| S(f)(\nu) \|_{\ell^p_\nu} \approx \| S(f)(x) \|_{L^p_x}$.
\item 
Let $\varphi$ be a function in $\calS(\R^n)$
with compact support.
Then, $| \varphi(D-\nu) f(x) |^{2}
\lesssim S( | \varphi(D-\nu) f |^{2} )(x)$
for any $f \in \calS(\R^{n})$, $\nu\in\Z^n$, and $x\in\R^n$.
\end{enumerate}
\end{lem}

The following is proved in \cite[Proposition 3.4]{KMT-arxiv-2} 
(see also \cite[Proposition 3.3]{KMT-arXiv}).

\begin{lem}\label{productLweak}
Let $2<p_1,p_2 < \infty$, 
$1/p_{1} + 1/p_{2} = 1/2$, 
and let 
$f_{j} \in \ell^{p_{j}, \infty}(\Z^d)$
be nonnegative sequences
for $j=1,2$. 
Then, 
\begin{align*}
\sum_{\nu_1, \nu_2 \in \Z^n} 
f_1(\nu_1) f_2(\nu_2) 
A_0(\nu_1+ \nu_2) 
\prod_{j=1,2} A_j(\nu_j)
\lesssim 
\|f_{1}\|_{ \ell^{p_{1},\infty} }
\|f_{2}\|_{ \ell^{p_{2},\infty} }
\prod_{j=0,1,2} \|A_j\|_{\ell^{2} }.
\end{align*}
\end{lem}

We end this section by mentioning a lemma which
can be found in Sugimoto \cite[Lemma 2.2.1]{sugimoto 1988 JMSJ}. 
The explicit proof is given in \cite[Lemma 4.4]{KMT-arxiv-2}.

\begin{lem}
\label{unifdecom}
There exist the functions $\kappa \in \calS(\R^n)$ and $\chi \in \calS(\R^n)$ satisfying that
$\supp \kappa \subset [-1,1]^n$, $\supp \widehat \chi \subset B_1$, 
$| \chi | \geq c > 0$ on $[-1,1]^n$
and
\begin{equation*}
\sum_{\nu\in\Z^n} \kappa (\xi - \nu) \chi (\xi - \nu ) = 1, \quad \xi \in \R^n.
\end{equation*}
\end{lem}


\section{Main results}\label{sectionMain}
\subsection{Key proposition} 
In this subsection we prove the following key proposition.
The proof is almost the same as in the proof of \cite[Proposition 5.1]{KMT-arxiv-2},
and the essential idea goes back to \cite{boulkhemair 1995}.

\begin{prop}
\label{main-prop}
Let $s_1, s_2 \in (0,\infty)$ satisfy $s_1+s_2=n/2$,
and let 
$R_{0}, R_{1}, R_{2} \in [1, \infty)$.
Suppose $\sigma$ is 
a bounded continuous function on 
$(\R^n)^{3}$ such that  
$\supp \calF \sigma \subset 
B_{R_0} \times B_{R_1} \times B_{R_2}$. 
Then 
\begin{equation*}
\left\| T_{ \sigma } \right\|_{H^{s_1} \times H^{s_2} \to (L^2, \ell^{1})}
\lesssim
\left( R_{0} R_{1} R_{2} \right)^{n/2}
\| \sigma \|_{ L^{2}_{ul} ((\R^{n})^{3})}. 
\end{equation*}
\end{prop}


\begin{proof}
We take a function $\theta \in \calS(\R^n)$ satisfying that 
$|\theta| \geq c >0$ on $Q=[-1/2,1/2)^n$ and 
$\supp \widehat \theta \subset B_1$.
Then, we have by Lemma \ref{lemequivalentamalgam}
\begin{equation*}
\left\| T_{ \sigma }( f_1, f_2 ) \right\|_{(L^2,\ell^{1})}
\approx
\left\| \theta (x-\mu) T_{ \sigma }( f_1, f_2 )(x) \right\|_{L^2_x(\R^n) \ell^{1}_{\mu}(\Z^n) },
\end{equation*}
and by duality
\begin{equation}
\label{saishodual}
\left\| T_{ \sigma }( f_1, f_2 ) \right\|_{(L^2,\ell^{1})}
\approx
\left\| \sup_{\|g\|_{L^2}=1} \left| \int_{\R^n} \theta (x-\mu) 
T_{ \sigma }( f_1, f_2 )(x) \, g(x) \, dx \right|\right\|_{ \ell^{1}_{\mu} }.
\end{equation}
Hence, in what follows we consider
\begin{equation*}
I = \int_{\R^n} \theta (x-\mu) T_{ \sigma }( f_1, f_2 )(x) \, g(x) \, dx
\end{equation*}
for any $\mu\in\Z^n$ and all $g\in L^2(\R^n)$.

Now, we rewrite the integral 
$I$ by the two steps below.
Firstly, by using Lemma \ref{unifdecom}, 
we decompose the symbol $\sigma$ as
\begin{align*}
\sigma ( x, \xi_1, \xi_2)
&=
\sum_{\boldsymbol{\nu}\in(\Z^n)^{2}} 
\sigma ( x, \xi_1, \xi_2 ) 
\kappa (\xi_{1} - \nu_{1}) \chi (\xi_{1} - \nu_{1} )
\kappa (\xi_{2} - \nu_{2}) \chi (\xi_{2} - \nu_{2} )
\\&=
\sum_{\boldsymbol{\nu}\in(\Z^n)^{2}} 
\sigma_{\boldsymbol{\nu}} ( x, \xi_1, \xi_2 ) 
\kappa (\xi_{1} - \nu_{1}) \kappa (\xi_{2} - \nu_{2}),
\end{align*}
where $\boldsymbol{\nu}=(\nu_1,\nu_2) \in (\Z^n)^2$ 
and we set
\[
\sigma_{\boldsymbol{\nu}} ( x, \xi_1, \xi_2 ) 
=
\sigma ( x, \xi_1, \xi_2 ) 
\chi (\xi_{1} - \nu_{1} )
\chi (\xi_{2} - \nu_{2} ).
\]
Denote
the Fourier multiplier operators 
$\kappa (D - \nu_{j})$
by $\square_{\nu_{j}}$, 
$j=1,2$.
Then, 
the integral $I$ is written as
\begin{equation}\label{decomposedI}
I=
\sum_{\boldsymbol{\nu}\in(\Z^n)^{2}} 
\int_{\R^n} \theta (x-\mu)
T_{ \sigma_{\boldsymbol{\nu}} } \big( \square_{\nu_1} f_1, \square_{\nu_2} f_2 \big)(x) 
g(x) 
\, dx .
\end{equation}
The idea of decomposing 
symbols by such
$\kappa$ and $\chi$ goes back to 
Sugimoto \cite{sugimoto 1988 JMSJ}.

Secondly, in the integral of \eqref{decomposedI}, 
we transfer the information of the Fourier transform of 
$\theta (\cdot-\mu) T_{ \sigma_{\boldsymbol{\nu}} }
( \square_{\nu_1}f_1,\square_{\nu_2}f_2 )$ to $g$.
Observe that
\begin{align*}
&
\calF \left[ T_{ \sigma_{\boldsymbol{\nu}} }( \square_{\nu_1}f_1, \square_{\nu_2}f_2 ) \right](\zeta)
\\&=
\frac{1}{(2\pi)^{2n}}
\int_{ (\R^{n})^{2} }
\big( \calF_0 \sigma_{\boldsymbol{\nu}} \big) \big( \zeta - (\xi_1 +\xi_2) , \xi_1, \xi_2 \big) 
\prod_{j=1,2} \kappa (\xi_{j}-\nu_{j}) \widehat{f_{j}}(\xi_{j})
\, d\xi_1 d\xi_2 .
\end{align*}
Then, combining this with the facts that
$\supp \calF_{0} \sigma_{\boldsymbol{\nu}} 
(\cdot, \xi_1, \xi_2 ) \subset B_{R_0}$
and $\supp \kappa(\cdot-\nu_{j}) 
\subset \nu_{j}+[-1,1]^n$, $j=1,2$, 
we see that
\begin{equation*}
\supp 
\calF \left[ T_{ \sigma_{\boldsymbol{\nu}} }( \square_{\nu_1}f_1, \square_{\nu_2}f_2 ) \right]
\subset 
\big\{ \zeta \in \R^n : | \zeta-(\nu_1 +\nu_2) | \lesssim R_0 \big\}.
\end{equation*}
Hence, since $\supp \widehat \theta \subset B_1$,
\begin{equation*}
\supp 
\calF \left[ \theta (\cdot-\mu)T_{ \sigma_{\boldsymbol{\nu}} }( \square_{\nu_1}f_1, \square_{\nu_2}f_2 ) \right]
\subset 
\big\{ \zeta \in \R^n : | \zeta-(\nu_1+\nu_2) | \lesssim R_0 \big\}
\end{equation*}
for any $\mu\in\Z^n$.
Taking a function $\varphi \in \calS(\R^n)$ 
satisfying that $\varphi = 1$ 
on $\{ \zeta\in\R^n: |\zeta|\lesssim 1\}$,
the integral $I$ given in \eqref{decomposedI} 
can be further rewritten as
\begin{equation}\label{decomposedFouriertransI}
I=
\sum_{\boldsymbol{\nu}\in(\Z^n)^2} \int_{\R^n} 
\theta (x-\mu) T_{ \sigma_{\boldsymbol{\nu}} }
( \square_{\nu_1}f_1, \square_{\nu_2}f_2 )(x) 
\, \varphi \left( \frac{D+\nu_1 + \nu_2}{R_0} \right) 
g (x) \, dx .
\end{equation}

Now, we shall actually estimate 
the newly rewritten integral $I$ 
in \eqref{decomposedFouriertransI}.
Since it holds from the facts
$\supp \calF_{1,2} \sigma (x, \cdot, \cdot)
\subset B_{R_1} \times B_{R_2}$
and $\supp \widehat{\chi} \subset B_1$
that
\[
\supp \calF_{1,2} \sigma_{\boldsymbol{\nu}} (x, \cdot, \cdot)
\subset  
B_{2R_1} \times B_{2R_2},
\]
we have
\begin{align*}
&
T_{ \sigma_{\boldsymbol{\nu}} }( \square_{\nu_1}f_1, \square_{\nu_2}f_2 )(x) 
\\&=
\frac{1}{(2\pi)^{2n}}
\int_{(\R^n)^2}
\big( \calF_{1,2} \sigma_{\boldsymbol{\nu}} \big) (x, y_1-x, y_2-x)
\prod_{j=1,2} \ichi_{B_{2R_{j}}}(x-y_{j}) \square_{\nu_{j}} f_{j} (y_{j}) 
\, dy_1 dy_2.
\end{align*}
Then, the Cauchy--Schwarz inequalities 
and the Plancherel theorem yield that
\begin{align*}
\left| T_{ \sigma_{\boldsymbol{\nu}} }( \square_{\nu_1}f_1, \square_{\nu_2}f_2 )(x) \right|
\lesssim
\big\| \sigma_{\boldsymbol{\nu}} (x,\xi_1, \xi_2) \big\|_{L^2_{\xi_1, \xi_2}}
\prod_{j=1,2}
\left\{ \left( \mathbf{1}_{B_{2R_{j}}} \ast \big| \square_{\nu_{j}} f_{j} \big|^2 \right) (x) \right\}^{1/2}
\end{align*}
for any ${\boldsymbol{\nu}} = (\nu_1, \nu_2) \in (\Z^n)^2$ and $x \in \R^n$.
From this, the integral $I$ is estimated as
\begin{align}
\label{Ibeforediscretize}
\begin{split}
|I| & \lesssim
\sum_{\boldsymbol{\nu}\in(\Z^n)^2} 
\int_{\R^n}|\theta (x-\mu)| \,
\big\| \sigma_{\boldsymbol{\nu}} (x,\xi_1, \xi_2) \big\|_{L^2_{\xi_1, \xi_2}}
\\
&\quad\times 
\prod_{j=1,2}
\left\{ \Big( \mathbf{1}_{B_{2R_{j}}} \ast \big|\square_{\nu_{j}} f_{j} \big|^2 \Big) (x) \right\}^{1/2}
\left| \varphi \left( \frac{D+\nu_1 + \nu_2}{R_0} \right) g (x) \right|
\, dx.
\end{split}
\end{align}
Next, we decompose the above integral over $x$
by using \eqref{cubicdiscretization}.
Then the inequality \eqref{Ibeforediscretize}
coincides with
\begin{align*}
|I| &\lesssim
\sum_{\nu_0\in\Z^n} 
\sum_{\boldsymbol{\nu}\in(\Z^n)^2} 
\int_{Q} |\theta (x+\nu_0-\mu)| \,
\big\| \sigma_{\boldsymbol{\nu}} (x+\nu_0,\xi_1, \xi_2) \big\|_{L^2_{\xi_1, \xi_2}}
\\
&\quad\times
\prod_{j=1,2}
\left\{ \Big( \mathbf{1}_{B_{2R_{j}}} \ast \big|\square_{\nu_{j}} f_{j} \big|^2 \Big) (x+\nu_0) \right\}^{1/2}
\left| \varphi \left( \frac{D+\nu_1 + \nu_2}{R_0} \right) g (x+\nu_0) \right|
\, dx.
\end{align*}
Observe that
$|\theta (x+\nu_0-\mu)| \lesssim \langle \nu_0-\mu \rangle^{-L}$ 
holds for any $x\in Q$ and some constant $L>0$ sufficiently large,
and
from Lemma \ref{propertiesofS} (4), (1), and (2) 
that
\begin{equation*}
\Big( \mathbf{1}_{B_{2R_{j}}} \ast 
\big|\square_{\nu_{j}} f_{j}\big|^2 \Big) (x+\nu_0)
\lesssim
S \Big( \mathbf{1}_{B_{2R_{j}}} \ast 
\big|\square_{\nu_{j}} f_{j} \big|^2 \Big) (\nu_0)
\end{equation*}
for $x\in Q$.
Then, by the
Cauchy--Schwarz inequality for the integral over $x$,
\begin{align*}
|I| &\lesssim
\sum_{\nu_0\in\Z^n} 
\sum_{\boldsymbol{\nu}\in(\Z^n)^2} 
\langle \nu_0-\mu \rangle^{-L}
\prod_{j=1,2}
\left\{ S \Big( \mathbf{1}_{B_{2R_{j}}} \ast 
\big|\square_{\nu_{j}} f_{j} \big|^2 \Big) (\nu_0) \right\}^{1/2}
\\
&\quad \times 
\int_{Q} 
\big\| \sigma_{\boldsymbol{\nu}} (x+\nu_0,\xi_1, \xi_2) \big\|_{L^2_{\xi_1, \xi_2}}
\left| \varphi \left( \frac{D+\nu_1 +\nu_2}{R_0} \right) g (x+\nu_0) \right|
\, dx
\\
&\lesssim
\sum_{\nu_0\in\Z^n} 
\sum_{\boldsymbol{\nu}\in(\Z^n)^2} 
\langle \nu_0-\mu \rangle^{-L}
\prod_{j=1,2}
\left\{ S \Big( \mathbf{1}_{B_{2R_{j}}} \ast \big|\square_{\nu_{j}} f_{j} \big|^2 \Big) (\nu_0) \right\}^{1/2}
\\
&\times
\big\| \sigma_{\boldsymbol{\nu}} (x+\nu_0,\xi_1, \xi_2) \big\|_{L^2_{\xi_1, \xi_2} L^2_x(Q)}
\left\| \varphi \left( \frac{D+\nu_1 + \nu_2}{R_0} \right) g (x+\nu_0) \right\|_{L^2_x(Q)}.
\end{align*}
Here, since the equivalence
\begin{equation*}
\sup_{\nu_0, \boldsymbol{\nu}}
\big\| \sigma_{\boldsymbol{\nu}} (x+\nu_0,\xi_1, \xi_2) \big\|_{L^2_{\xi_1, \xi_2} L^2_x(Q)}
\approx 
\| \sigma \|_{L^2_{ul} ((\R^n)^{3})}
\end{equation*}
holds, we have
\begin{align}
\label{Iafterdiscretize}
\begin{split}
|I|
&\lesssim
\| \sigma \|_{L^2_{ul}}
\sum_{\nu_0\in\Z^n} 
\sum_{\boldsymbol{\nu}\in(\Z^n)^2} 
\langle \nu_0-\mu \rangle^{-L}
\\
&\times
\prod_{j=1,2}
\left\{ S \Big( \mathbf{1}_{B_{2R_{j}}} \ast 
\big|\square_{\nu_{j}} f_{j} \big|^2 \Big) (\nu_0) \right\}^{1/2}
\left\| \varphi \left( \frac{D+\nu_1 + \nu_2}{R_0} \right) g (x+\nu_0) \right\|_{L^2_x(Q)}.
\end{split}
\end{align}
In what follows, we write
each summand above by
\begin{align*}
A_j (\nu_j,\nu_0)&=\left\{ S \left( \mathbf{1}_{B_{2R_j}} \ast \big| \square_{\nu_j} f_j \big|^2 \right)(\nu_0) \right\}^{1/2}, \quad
j=1,2,
\\
A_0(\tau,\nu_0)
&=
\left\|\varphi \left( \frac{D+\tau}{R_0} \right) g(x+\nu_0) \right\|_{L^2_x(Q)}.
\end{align*}
Then, the inequality \eqref{Iafterdiscretize} 
is written by
\begin{equation}
\label{IwithII}
|I|\lesssim
\| \sigma \|_{L^2_{ul}}
\, \II
\end{equation}
with
\begin{equation*}
\II=
\sum_{\nu_0\in\Z^n} 
\langle \nu_0-\mu \rangle^{-L}
\sum_{\boldsymbol{\nu}\in(\Z^n)^2} 
A_0(\nu_1 + \nu_2, \nu_0)
\prod_{j=1,2} A_j (\nu_j,\nu_0).
\end{equation*}

Now, we shall estimate $\II$. 
By applying Lemma \ref{productLweak}
with $s_1 + s_2 = n/2$
to the sum over $\boldsymbol{\nu}$, we have
\begin{align*}
\II &=
\sum_{\nu_0\in\Z^n} 
\langle \nu_0-\mu \rangle^{-L}
\sum_{\boldsymbol{\nu}\in(\Z^n)^2} 
\langle \nu_1 \rangle^{-s_1} \langle \nu_2 \rangle^{-s_2}
A_0(\nu_1 + \nu_2, \nu_0)
\prod_{j=1,2} \langle \nu_j \rangle^{s_j} A_j (\nu_j,\nu_0)
\\&\lesssim 
\sum_{ \nu_0 \in \Z^{n} }
\langle \nu_0-\mu \rangle^{-L}
\| A_0 (\tau,\nu_0) \|_{\ell^2_{\tau}}
\prod_{j=1,2} 
\| \langle \nu_j \rangle^{s_j} A_j (\nu_j,\nu_0) \|_{\ell^2_{\nu_j}},
\end{align*}
since $\langle \nu_j \rangle^{-s_j} 
\in \ell^{n/s_j,\infty}(\Z^n)$, $j=1,2$.
Then, we use the H\"older inequality 
to the sum over $\nu_0$ to have
\begin{equation}\label{IIHolder}
\II
\lesssim 
\| A_0 (\tau,\nu_0) \|_{ \ell^2_{\tau} \ell^{2}_{\nu_0} }
\prod_{j=1,2} 
\| \langle \nu_0-\mu \rangle^{-L/2} \langle \nu_j \rangle^{s_j} A_j (\nu_j,\nu_0) \|_{ \ell^2_{\nu_j} \ell^{4}_{\nu_0}}. 
\end{equation}
Here, the norm of $A_0$ in \eqref{IIHolder}  
is estimated by the Plancherel theorem as follows:
\begin{align}\label{A0result}
\begin{split}
&\| A_0 (\tau,\nu_0) \|_{ \ell^2_{\tau} \ell^{2}_{\nu_0} }
=
\left\|\varphi \left( \frac{D+\tau}{R_0} \right) g(x+\nu_0) \right\|_{L^2_x(Q) \ell^2_{\tau} \ell^{2}_{\nu_0}}
\\
&=\left\|\varphi \left( \frac{D+\tau}{R_0} \right) g(x)
\right\|_{L^2_x(\R^n) \ell^2_{\tau}}
\approx
\left\|\varphi \left( \frac{\zeta +\tau}{R_0} \right) \widehat g(\zeta) \right\|_{L^2_\zeta(\R^n) \ell^2_{\tau}}
\approx 
R_{0}^{n/2} \| g \|_{L^2},
\end{split}
\end{align}
where we used that 
$\| \varphi ( \frac{\zeta +\tau}{R_0} ) \|_{\ell^2_\tau} 
\approx R_0^{n/2}$ for any $\zeta \in \R^n$.
Hence, by collecting \eqref{IwithII}, \eqref{IIHolder}, and \eqref{A0result} we have
\begin{equation}
\label{resultofI}
|I|\lesssim
R_{0}^{n/2} \| \sigma \|_{L^2_{ul}}
\, \| g \|_{L^2}
\, \prod_{j=1,2} 
\| \langle \nu_0-\mu \rangle^{-L/2} \langle \nu_j \rangle^{s_j} A_j (\nu_j,\nu_0) \|_{ \ell^2_{\nu_j} \ell^{4}_{\nu_0}}.
\end{equation}

We substitute \eqref{resultofI} into \eqref{saishodual}, 
and then use the H\"older inequality to $\ell^1_{\mu}$.
Then,
\begin{equation}\label{atodual}
\left\| T_{ \sigma }( f_1, f_2 ) \right\|_{(L^2,\ell^1)}
\lesssim
R_{0}^{n/2} \| \sigma \|_{L^2_{ul}}
\, \prod_{j=1,2} 
\left\| \langle \nu_0-\mu \rangle^{-L/2} \langle \nu_j \rangle^{s_j} A_j (\nu_j,\nu_0) \right\|_{ \ell^2_{\nu_j} \ell^{4}_{\nu_0} \ell^{2}_{\mu}}.
\end{equation}
To achieve our goal, we shall estimate the norm of $A_j$ in \eqref{atodual}.
We have
\begin{align*}
&
\left\| \langle \nu_0-\mu \rangle^{-L/2} \langle \nu_j \rangle^{s_j} A_j (\nu_j,\nu_0) \right\|_{ \ell^2_{\nu_j} \ell^{4}_{\nu_0} \ell^{2}_{\mu}}
\\
&=
\bigg\| \langle \nu_0-\mu \rangle^{-L} 
S \Big( \mathbf{1}_{B_{2R_j}} \ast 
\big| \langle \nu_j \rangle^{s_j} 
\square_{\nu_{j}} f_{j} \big|^2 \Big)(\nu_0) 
\bigg\|_{\ell^{1}_{\nu_{j}}\ell^{2}_{\nu_0} \ell^{1}_{\mu} }^{1/2}
\\
&=
\bigg\| \langle \nu_0 \rangle^{-L} 
S \Big( \mathbf{1}_{B_{2R_j}} \ast 
\big\| \langle \nu_j \rangle^{s_j} \square_{\nu_{j}} 
f_{j} \big\|_{\ell^2_{\nu_{j}}}^2 \Big)(\nu_0+\mu) 
\bigg\|_{\ell^{2}_{\nu_0} \ell^{1}_{\mu} }^{1/2}
=(\ast\ast).
\end{align*}
Then, by using the embedding 
$\ell^1_{\nu_0} \hookrightarrow \ell^2_{\nu_0}$,
Lemma \ref{propertiesofS} (3),
and the boundedness of the operator 
$S$ on $L^1$,
we have
\begin{align*}
(\ast\ast)
\lesssim
\left\| 
\mathbf{1}_{B_{2R_j}} \ast 
\big\| \langle \nu_j \rangle^{s_j} \square_{\nu_{j}} 
f_{j} \big\|_{\ell^2_{\nu_{j}}}^2\right\|_{ L^{1}_{x} }^{1/2}
\lesssim
R_j^{n/2}
\left\| \langle \nu_j \rangle^{s_j} \square_{\nu_{j}} 
f_{j} \right\|_{\ell^2_{\nu_{j}} L^{2}_{x} },
\end{align*}
where the first inequality holds if 
$L$ is suitably large. 
Here, by the Plancherel theorem
\begin{align*}
\left\| \langle \nu_j \rangle^{s_j} \square_{\nu_{j}} f_{j} 
\right\|_{\ell^2_{\nu_{j}} L^{2}_{x} }
&\approx
\left\| \langle \nu_j \rangle^{s_j} \kappa(\xi - \nu_{j}) 
\widehat{f_{j}}(\xi) \right\|_{\ell^2_{\nu_{j}} L^{2}_{\xi} }
\\&\approx
\left\| \langle \xi \rangle^{s_j} \kappa(\xi - \nu_{j}) 
\widehat{f_{j}}(\xi) \right\|_{\ell^2_{\nu_{j}} L^{2}_{\xi} }
\approx
\| f_{j} \|_{ H^{s_j} },
\end{align*}
where, we used that
$\sum_{\nu_j \in \Z^n} |\kappa (\cdot - \nu_j) |^2 \approx 1$
to have the last equivalence.
Therefore,
\begin{align}\label{resultAj}
&
\left\| \langle \nu_0-\mu \rangle^{-L/2} \langle \nu_j \rangle^{s_j} A_j (\nu_j,\nu_0) \right\|_{ \ell^2_{\nu_j} \ell^{4}_{\nu_0} \ell^{2}_{\mu}}
\lesssim
R_j^{n/2}
\left\| f_{j} \right\|_{ H^{s_j} }.
\end{align}

Substituting \eqref{resultAj} into \eqref{atodual}, we obtain
\begin{equation*}
\left\| T_{ \sigma }( f_1, f_2 ) \right\|_{(L^2,\ell^1)}
\lesssim
\left( R_{0} R_{1} R_{2} \right)^{n/2} 
\| \sigma \|_{L^2_{ul}}
\prod_{j=1,2}  \left\| f_{j} \right\|_{ H^{s_j} },
\end{equation*}
which completes the proof.
\end{proof}

\subsection{Proof of Theorem \ref{main-thm}}
\label{subsection-Proof-Main-1} 

From Proposition \ref{main-prop}, 
we shall deduce Theorem \ref{main-thm}.

\begin{proof}[Proof of Theorem \ref{main-thm}]
We decompose the symbol $\sigma$ 
by a partition $\{ \varphi (\cdot - k ) \}_{k\in\Z^n}$ 
stated in Section \ref{secmodsp}:
\begin{align*}
\sigma (x, \xi_1, \xi_2)
&=
\sum_{ (k_0,k_1,k_2) \in (\Z^n)^{3}} 
\varphi (D_x - k_0)
\varphi (D_{\xi_1} - k_1)
\varphi (D_{\xi_2} - k_2)
\sigma (x, \xi_1, \xi_2).
\\&=
\sum_{ \boldsymbol{k} \in (\Z^n)^{3}} 
\square_{\boldsymbol{k}} 
\sigma (x, \xi_1, \xi_2)
\end{align*}
with
\[
\square_{\boldsymbol{k}} 
\sigma (x, \xi_1, \xi_2)
=
\varphi (D_x - k_0)
\varphi (D_{\xi_1} - k_1)
\varphi (D_{\xi_2} - k_2)
\sigma (x, \xi_1, \xi_2),
\]
where $\boldsymbol{k}=(k_0,k_1,k_2) \in (\Z^n)^3$.
Here, we observe that
\begin{equation*}
\varphi (D-k) f (x) 
= e^{i x \cdot k} \varphi (D) \left[ M_{k} f \right] (x), \quad
k \in \Z^n,
\end{equation*}
where $M_{k} f (z) = e^{-ik \cdot z} f (z)$.
From this, we see that 
\begin{align*}
\square_{\boldsymbol{k}} 
\sigma (x, \xi_1, \xi_2)
=
e^{i x \cdot k_0} e^{i \xi_1 \cdot k_1} e^{i \xi_2 \cdot k_2} \,
\square_{(0,0,0)}
\left[ (M_{k_0} \otimes M_{k_1} \otimes M_{k_2} ) \sigma \right] (x,\xi_1, \xi_2).
\end{align*}
Hence,
\begin{align*}
T_{ \square_{\boldsymbol{k}} \sigma } (f_1, f_2)
= e^{i x \cdot k_0} \,
T_{\square_{(0,0,0)} 
	\left[ (M_{k_0} \otimes M_{k_1} \otimes M_{k_2} ) \sigma \right] }
\big(f_1(\cdot+k_1), f_2(\cdot+k_2)\big).
\end{align*}
Since
\begin{equation*}
\supp \calF \left[ \square_{(0,0,0)} 
\left[ (M_{k_0} \otimes M_{k_1} \otimes M_{k_2} ) \sigma \right] \right]
\subset [-1,1]^{3n},
\end{equation*}
we have 
by Proposition \ref{main-prop} 
and the translation invariance of the Sobolev space
\begin{align*}
\| T_{ \square_{\boldsymbol{k}} \sigma } \|_{H^{s_1} \times H^{s_2} \to (L^2,\ell^1)}
\lesssim
\left\| \square_{(0,0,0)}\left[ (M_{k_0} \otimes M_{k_1} \otimes M_{k_2} ) \sigma \right] \right\|_{L^2_{ul}}
=
\left\| \square_{\boldsymbol{k}} \sigma \right\|_{L^2_{ul}}.
\end{align*}
Furthermore, we have the equivalence
\begin{equation}\label{L2ulLinfty}
\| \square_{\boldsymbol{k}} \sigma \|_{L^2_{ul}}
\approx 
\| \square_{\boldsymbol{k}} \sigma \|_{L^\infty}
\end{equation}
with implicit constants independ of $\boldsymbol{k}$
(see Remark \ref{remL2ulLinfty} below).
Therefore,
\begin{align*}
\| T_\sigma \|_{H^{s_1} \times H^{s_2} \to (L^2,\ell^1)}
&\leq
\sum_{ \boldsymbol{k} \in (\Z^{n})^3 }
\left\| T_{ \square_{\boldsymbol{k}} \sigma } \right\|_{H^{s_1} \times H^{s_2} \to (L^2,\ell^1)}
\\&\lesssim
\sum_{ \boldsymbol{k} \in (\Z^{n})^3 }
\| \square_{\boldsymbol{k}} \sigma \|_{L^2_{ul}}
\approx
\| \sigma \|_{M^{\infty,1}},
\end{align*}
which completes the proof of Theorem \ref{main-thm}.
\end{proof}


\begin{rem}\label{remL2ulLinfty}
The equivalence \eqref{L2ulLinfty}
was already pointed out by Boulkhemair
\cite[Appendix A.1]{boulkhemair 1995}.
However, for the reader's convenience,
we give a proof of \eqref{L2ulLinfty}.
The way of the proof here is essentialy 
the same as was given in
\cite[Appendix A.1]{boulkhemair 1995}.

Let $\{ \varphi (\cdot - k ) \}_{k\in\Z^d}$ be
a partition on $\R^d$ stated Section \ref{secmodsp}.
We take a function $\phi \in \calS(\R^d)$
satisfying that
$\phi = 1$ on $[-1,1]^d$.
Then, 
\begin{align*}
\varphi(D-k) f(x) 
&=
\phi(D-k) \varphi(D-k) f(x) 
\\&=
\int_{\R^d} e^{i(x-y)\cdot k} \check \phi(x-y) 
\, \varphi(D-k) f(y) \,dy .
\end{align*}
By the Cauchy--Schwarz inequality, we have
\begin{align*}
| \varphi(D-k) f(x) |
&\lesssim
\int_{\R^d} \langle x-y \rangle^{-d-1} 
| \varphi(D-k) f(y) | \,dy 
\\&\lesssim
\left( \int_{\R^d} \langle x-y \rangle^{-d-1} 
| \varphi(D-k) f(y) |^2 \,dy \right)^{1/2}
=(\dag)
\end{align*}
for any $k \in \Z^d$ and $x\in\R^d$.
As in \eqref{cubicdiscretization},
we decompose the integral and have
\begin{align*}
(\dag)
&\approx
\left( \sum_{\nu \in \Z^d } \int_{ [-1/2,1/2)^d } \langle x-\nu \rangle^{-d-1} 
| \varphi(D-k) f(y+\nu) |^2 \,dy \right)^{1/2}
\\&\leq 
\| \varphi(D-k) f \|_{L^2_{ul}(\R^d)}
\left( \sum_{\nu \in \Z^d }\langle x-\nu \rangle^{-d-1} \right)^{1/2}
\approx
\| \varphi(D-k) f \|_{L^2_{ul}(\R^d)}
\end{align*}
for any $k \in \Z^d$ and $x\in\R^d$.
Therefore, we obtain the inequality 
$\gtrsim$ in \eqref{L2ulLinfty}.
The opposite inequality is obvious,
so that we have the equivalence \eqref{L2ulLinfty}.
\end{rem}

\section*{Acknowledgments}
The author expresses his special thanks
to Professor Akihiko Miyachi and Professor Naohito Tomita
for a lot of fruitful suggestions and encouragement.



\end{document}